\documentclass[12pt]{amsart}
\usepackage{graphicx}
\usepackage{graphics}
\usepackage{amsmath}
\usepackage{amscd}
\usepackage{latexsym}
\usepackage{amssymb}
\usepackage{amsthm}
\usepackage{dsfont}
\usepackage{fancyhdr}
\usepackage[shortlabels]{enumitem}

\usepackage{tikz}
\usetikzlibrary{arrows,cd,matrix,positioning,decorations.pathmorphing}
\usepackage{setspace}
\tikzset{snake arrow/.style=
{->,
decorate,
decoration={snake,amplitude=.4mm,segment length=2mm,post length=1mm}},
}

\usepackage{blindtext}
\usepackage{scrextend}
\addtokomafont{labelinglabel}{\sffamily}

\usepackage{hyperref}  
\hypersetup{
  colorlinks   = true, 
  urlcolor     = blue, 
  linkcolor    = blue, 
  citecolor    = blue  
}

\begin{document}

\textwidth 6.2in
\textheight 7.6in
\evensidemargin.75in
\oddsidemargin.75in

\newtheorem{Thm}{Theorem}
\newtheorem{Lem}[Thm]{Lemma}
\newtheorem{Cor}[Thm]{Corollary}
\newtheorem{Prop}[Thm]{Proposition}
\newtheorem{Rm}{Remark}
\newtheorem{Qu}{Question}
\newtheorem{Def}{Definition}
\newtheorem{Exm}{Example}

\def\a{{\mathbb a}}
\def\C{{\mathbb C}}
\def\A{{\mathbb A}}
\def\B{{\mathbb B}}
\def\D{{\mathbb D}}
\def\E{{\mathbb E}}
\def\R{{\mathbb R}}
\def\P{{\mathbb P}}
\def\S{{\mathbb S}}
\def\Z{{\mathbb Z}}
\def\O{{\mathbb O}}
\def\H{{\mathbb H}}
\def\V{{\mathbb V}}
\def\Q{{\mathbb Q}}
\def\Cn{${\mathcal C}_n$}
\def\CM{\mathcal M}
\def\CG{\mathcal G}
\def\CH{\mathcal H}
\def\CT{\mathcal T}
\def\CF{\mathcal F}
\def\CA{\mathcal A}
\def\CB{\mathcal B}
\def\CD{\mathcal D}
\def\CP{\mathcal P}
\def\CS{\mathcal S}
\def\CZ{\mathcal Z}
\def\CE{\mathcal E}
\def\CL{\mathcal L}
\def\CV{\mathcal V}
\def\CW{\mathcal W}
\def\IC{\mathbb C}
\def\IF{\mathbb F}
\def\IK{\mathcal K}
\def\IL{\mathcal L}
\def\IP{\bf P}
\def\IR{\mathbb R}
\def\IZ{\mathbb Z}

\title{On Smooth $4$-dimensional Poincar\'e conjecture}
\author{Selman Akbulut}
\address{G\"{o}kova Geometry Topology Institute,  Mu\u{g}la, Turkiye}
\email{akbulut.selman@gmail.com}
\date{\today}
\begin{abstract} 
 \hspace{-.05in} 
 A proof of 4-dimensional smooth Poincar\'e Conjecture.
\end{abstract}

\date{}
\maketitle

\setcounter{section}{-1} 

\vspace{-.2in}

\section{Statement and a proof}  
It is known that homotopy equivalent simply connected closed smooth $4$-manifolds are $h$-cobordant to each other, and they differ from each other by corks \cite{a4}, \cite{m}. For the background material, and the techniques used in this paper we refer reader to  \cite{a1} and \cite{a2}.

\begin{Thm}\label{main} 
Every smooth homotopy $4$-sphere $\Sigma^{4}$ (i.e. closed smooth $4$-manifold homotopy equivalent to $S^{4}$) is diffeomorphic to $S^{4}$. 
\end{Thm}

\begin{proof} We know that every such $\Sigma^{4}$ is smoothly $h$-cobordant to $S^{4}$, through a cobordism $Z^{5}$ consisting  of only $2$- and $3$-handles.  Let  $Z^{5}$ be such an $h$-cobordism from $S^{4}$ to $\Sigma^{4}$, obtained by attaching  $2$- handles $\{h^{2}_{j}\}$  to $S^{4}$, then attaching $3$-handles top of them. The $3$-handles upside down are duals to the $2$-handles $\{k^{2}_{j}\}$, attached to $\Sigma^{4}$. This decomposes $Z^{5}$ into union of two parts, meeting in the middle $4$-manifold $X^{4}$. 

$$Z^{5}= [S^{4} \cup h^{2}_{j}] \smile_{X} [\Sigma^{4}\cup k^{2}_{j}]$$

\begin{figure}[ht]  \begin{center}
\includegraphics[width=.4\textwidth]{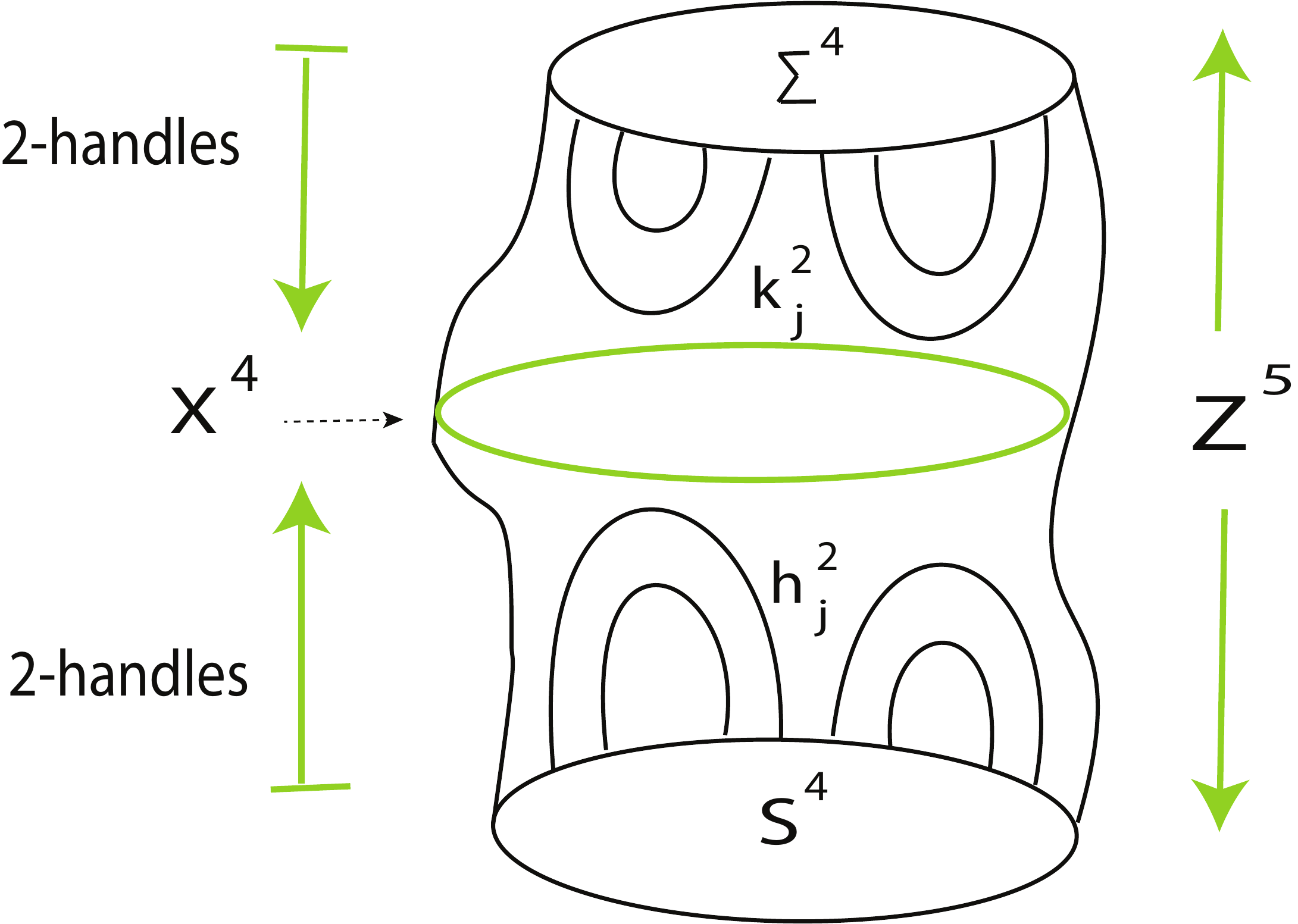}   
\caption{h-cobordism}   \label{k1}
\end{center}
\end{figure}

Since $S^{4}$ is simply connected, $X$ is diffeomorphic to  $\#_{n}S^{2}\times S^{2}$. By looking $Z^{5}$ from the other side, since $\Sigma^{4}$ is simply connected, we see $X$ is diffeomorphic to $(\# _{n} S^{2}\times S^{2}) \# \Sigma$.
Therefore $X$ contains two families of disjointly imbedded $2$-spheres $\CA =\{\sqcup A_{i}\}_{i=1}^{n}$ and $\CB=\{\sqcup B_{j}\}_{j=1}^{n}$, which are the belt spheres of $2$-handles $\{h^{2}_{j}\}$ and $\{k^{2}_{j}\}$, respectively.  Since $Z^{5}$ is homotopy product, they are imbedded with trivial normal bundles, and we can arrange algebraic intersection numbers $A_{i} . B_{j}=\delta_{ij}$. 

\begin{figure}[ht]  \begin{center}
\includegraphics[width=.65\textwidth]{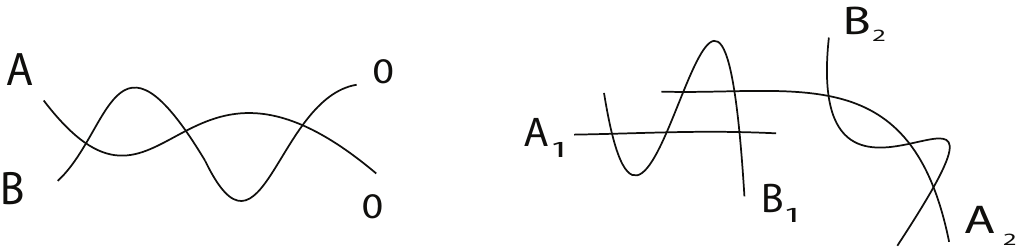}   
\caption{Various protocorks}   \label{y1}
\end{center}
\end{figure}

Let us denote the union of the tubular neighborhoods of these spheres in $X$ by $\V\subset X$. Clearly  surgering $X$ along the  $\CA$-spheres gives $S^{4}$, and surgering $X$ along the $\CB$-spheres gives $\Sigma^{4}$ (because these surgeries undo the affect of the $2$-handles). 
Now denote the manifold obtained from $\V$ by surgering  $\CA$-spheres by $\V(\CA)$; and  denote the manifold obtained from $\V$ by surgering  $\CB$-spheres by $\V(\CB)$. Clearly we have the inclusions $\V(\CA)\subset S^{4}$ and  $\V(\CB)\subset \Sigma^{4}$ . Note that $\partial \V(\CA)=\partial \V(\CB)$, and we have a diffeomorphism 
  $X - \V(\CA) \approx X - \V(\CB)$. We call $\V(\CA)$ a {\it protocork}, and call the following operation ``protocork twisting'' of $S^{4}$ along $\V(\CA)$. 
  $$f:  S^{4} \rightsquigarrow (S^{4} -V(\CA)) \smile V(\CB) =\Sigma^{4}$$ 
  
\noindent This corresponds to ``zero-dot exchange operation'' of corks.   Let us denote the protocork shown in the left picture of Figure~\ref{y1} by $\P_{n}$ ($2n+1$ intersection points). Handlebody of $\P_{n}$ can be described as in Figure~\ref{y2}.  The first cork appeared in \cite{a3} then generalized in \cite{m}. Protocorks are simpler objects than corks, which are basic building blocks of corks. Recently they were used by Ladu as tools in geometric analysis \cite{l}.  

\begin{figure}[ht]  \begin{center}
\includegraphics[width=.35\textwidth]{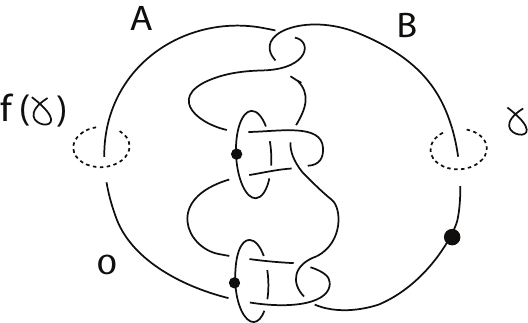}   
\caption{$\P_{n}$ when $n=1$}   \label{y2}
\end{center}
\end{figure}

\newpage

Next, starting with any imbedding of a protocork $\V(\A)\subset S^{4}$, we will describe the handlebody of the protocork twisting of $S^{4}$ along $\V(\A)$. For simplicity, we will first take $\V(\A) = \P_{n}$, then indicate how to modify the argument for the general case. We proceed by two steps. 
\vspace{.1in}
\begin{itemize}

\item[(1)] Draw the handlebody picture of the complement $C:=S^{4}-\P_{n}$.

\item[(2)]  Glue $\P_{n}$ to $C $ by the involution 
$f:\partial \P_{n}\to \partial \P_{n}$.

\end{itemize}

\vspace{.1in}

To simplify this process we first carve $\P_{n}$ across the meridianal $2$-disk $D$ of its $2$-handle (which $f(\gamma)$ bounds). Then notice that, up to $3$-handles this gives the identification $S^{4} = C\smile_{\partial} h^{2}_{f(\gamma)}$, because $N(D)$ upside-down is just the $2$-handle $h^{2}_{f(\gamma)}$ attached to $C$, and $\P_{n}-N(D) \approx \#_{k}B^{3}\times S^{1}$ can be viewed as $3$-handles attached to $C\smile h^{2}_{f(\gamma)}$. Hence, up to $3$-handles, gluing $\P_{n}$ to $C$ by  $f$  can be described by $C\smile h^{2}_{\gamma}$. 

  \begin{figure}[ht]  \begin{center}
\includegraphics[width=.62\textwidth]{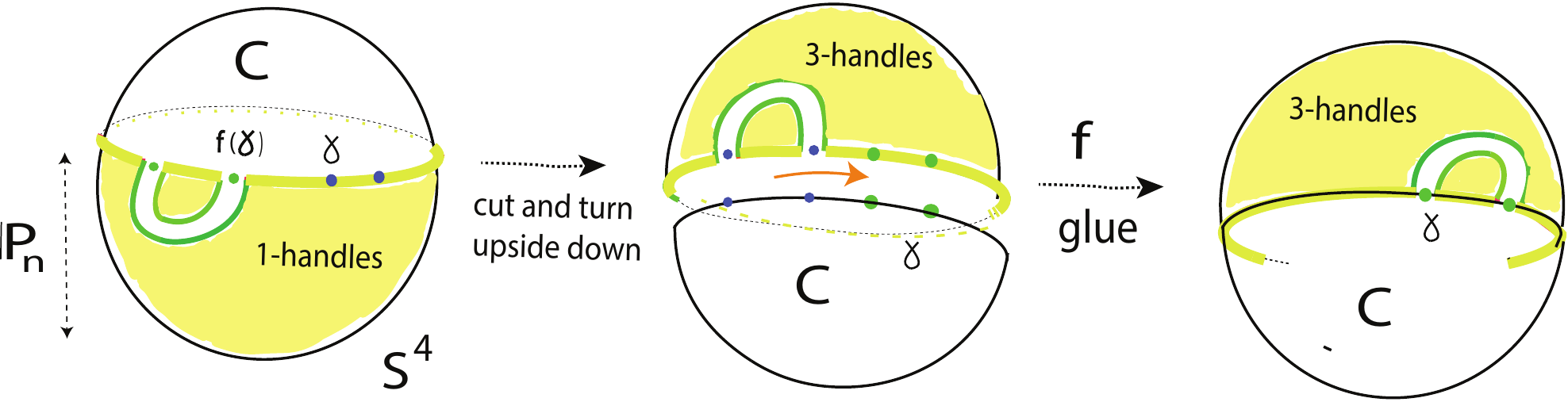}
\caption{$C\smile h_{f(\gamma)} \longrightarrow C\smile h_{\gamma}$}   \label{p1}
\end{center}
\end{figure}

From any imbedding $\P_{n}\subset S^{4}$, we can construct a handlebody of  $C= S^{4}-\P_{n}$ in two steps, which amounts to turning $S^{4}$ upside down:

\begin{figure}[ht]  \begin{center}
\includegraphics[width=.44\textwidth]{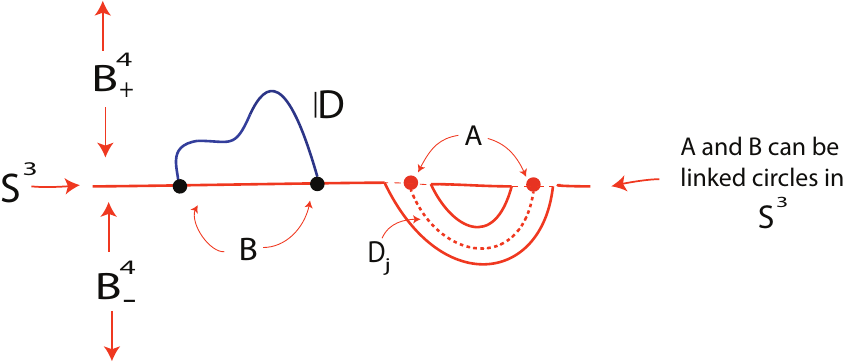}
\caption{$S^{4}= B_{-}^{4}\smile B_{+}^{4}$}   \label{y7b}
\end{center}
\end{figure}

(a) Isotope the $1$-handles of $\P_{n}$ into standard position $\#_{k} B^{3}\times S^{1} $, then take its complement in $S^{4}$, which is $\#_{k} S^{2}\times D^{2}$.  

\vspace{.1in}

(b) Carve this $\#_{k} S^{2}\times D^{2}$ along the core $\D$ of the $2$-handle of $\P_{n}$.

\vspace{.1in}

 \begin{Rm} $1$-handles $\#_{k} B^{3}\times S^{1}$ of  $ \P_{n}$ are just carved out  $2$-discs from the $4$-ball $B_{-}^{4}$ (Figure~\ref{y7b}), hence their complement in $S^{4}$ is the standard  $\#_{k} S^{2}\times B^{2} \subset S^{4}$. Here $B_{+}^{4} = S^{4}-B_{-}^{4}$, and $\D$ is the core of the $2$-handle of $\P_{n}$ attached to carved $B_{-}^{4}$. Note that even though $\D$ is just the core of the $2$-handle, its complement in $S^{4}$ (viewed upside down) could be complicated looking slice disk complement. That is, $2$-handle $N(\D)$ of $\P_{n}$, is attached to $\#_{k} B^{3}\times S^{1}$, and $S^{4}-\P_{n} \subset S^{4}$ will be the complement of the slice disk $\D$ in $\#_{k} S^{2}\times B^{2}$. So all possible embeddings $\P_{n}\subset S^{4}$ are determined by the choice of this slice disk $\D$. 
 \end{Rm}

 Even though the slice disk $\D$ is not unique, from the linking pattern of the circles $A$ and $B$ in Figure~\ref{y2}, we can determine intersections of $\D$ with the $2$-handles $\#_{k} S^{2}\times B^{2}$ of $S^{4}-\P_{n}$.  Figure~\ref{y5a} summarizes what we have done so far. The doted circle labeled by $f(\gamma)$ indicates carving of $\P_{n}$ along the meridianal disk $D$ (looking from outside its just a $2$-handle $h^{2}_{f(\gamma)} $). Then going from the second to the last picture of Figure~\ref{y5a} describes the protocork twisting, i.e. 
 $$(S^{4}-\P_{n})\smile h^{2}_{f(\gamma)} \to 
 (S^{4}-\P_{n})\smile h^{2}_{(\gamma)}$$

 \begin{figure}[ht]  \begin{center}
\includegraphics[width=.83\textwidth]{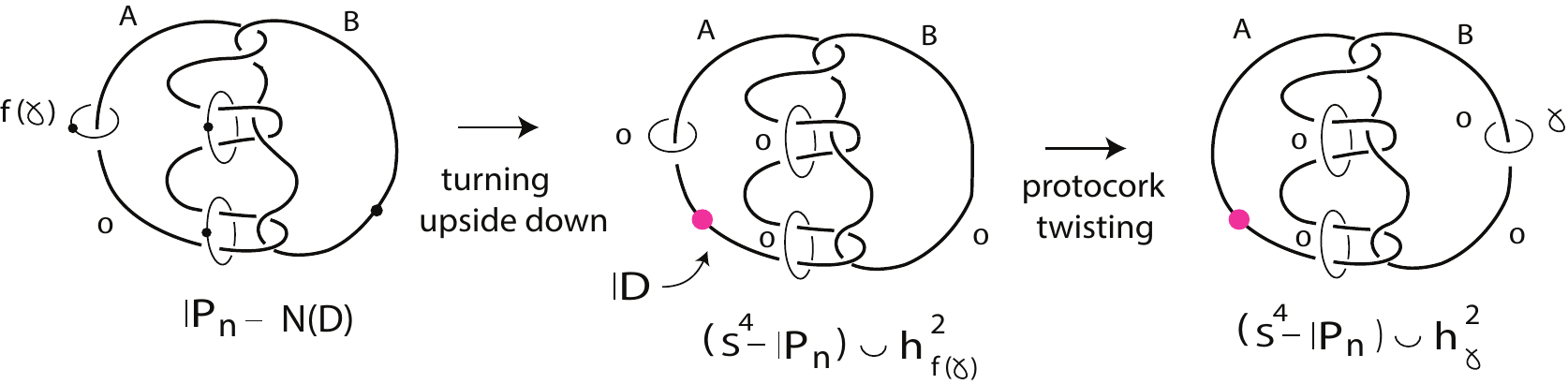}
\caption{Protocork twisting of $S^{4}$ along $\P_{n}, n=1$}   \label{y5a}
\end{center}
\end{figure}

 Then by the isotopy in Figure~\ref{y5b}, followed by the indicated handle slide, we arrive to the last picture of Figure~\ref{y5b}. After cancelling the two isolated $0$-framed unknots (i.e. $S^{2}\times B^{2}$'s) with $3$-handles,  we are now left with a chain of three linked circles: $A, B$, and $C$; where $A$ is the boundary of some carved $2$-disk in $S^{2}\times B^{2}$ (bounding a meridian $ \partial B^{2}$). 
 
 \vspace{.1in}
 
 Note that, by construction, this $2$-disk is now slid over to the other side of $C$, i.e. it doesn't go over the $2$-handle $C$  (compare this to the sliding operation in Section 1.4 of \cite{a1}). Hence by \cite{g} the red dotted circle, with a $0$-framed circle $B$ linking it, gives $B^{4}$. More specifically, any imbedded $2$-disk in $S^{2}\times B^{2}$ whose boundary is $p \times \partial B^{2}$ is isotopic to $p\times B^{2}$. Therefore the last picture of Figure~\ref{y5b} is just $B^{4} $ with a $2$-handle $C$ attached to it (with $0$-framing); and its boundary is diffeomorphic to $S^{1}\times S^{2}$. Hence by the {\it Property R} it must be $S^{2}\times B^{2}$, so this final manifold must be $S^{4}$, up to $3$-handles.  Therefore Theorem~\ref{main} in the case of  the protocork $ \V_{n}(\A) =\P_{n}$ has now been proven.

 \begin{figure}[ht]  \begin{center}
\includegraphics[width=.73\textwidth]{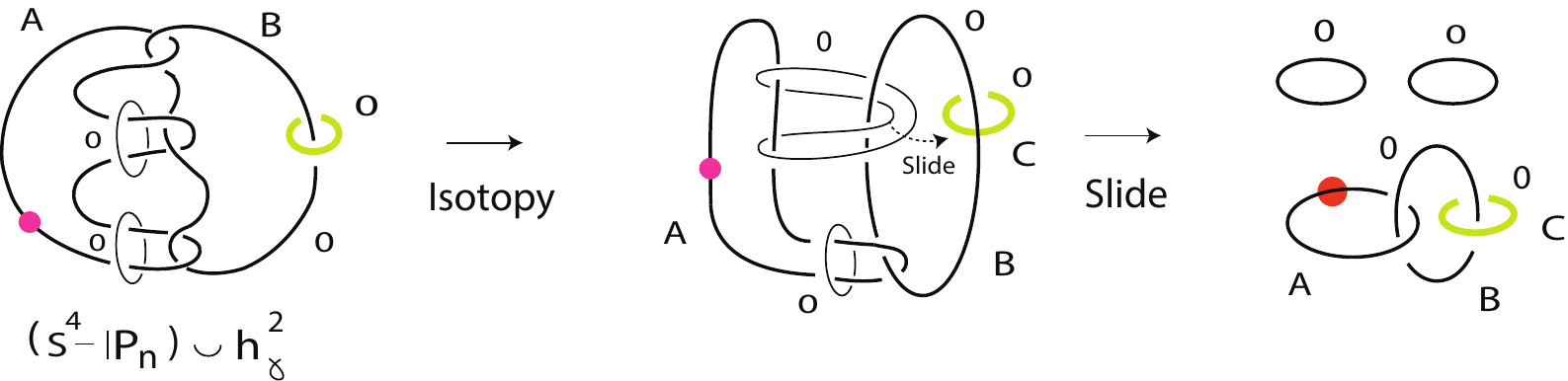}
\caption{Protocork twist continued}   \label{y5b}
\end{center}
\end{figure}

  \section {General case of $\P_{n}$ and further remarks}
  
  \vspace{.1in}

The protocork $\P_{n}$  can be identified with manifold called $L_{n}^{0}$ in \cite{ay}. Namely there is the identification of Figure~\ref{a1}.
We leave checking this as an exercise to the reader (Hint: Cancel the large dotted circle with the large 0-framed circle of $\P_{n}$, along the way slide dotted circles over each other if you have to). $L_{n}^{0}$ is the right picture of Figure~\ref{a1}, which is the manifold obtained by carving $B^{4}$ along the indicated slice disks. Figure~\ref{a6}  indicates a proof in the case  $n=1$. 

\begin{figure}[ht]  \begin{center}
\includegraphics[width=.65\textwidth]{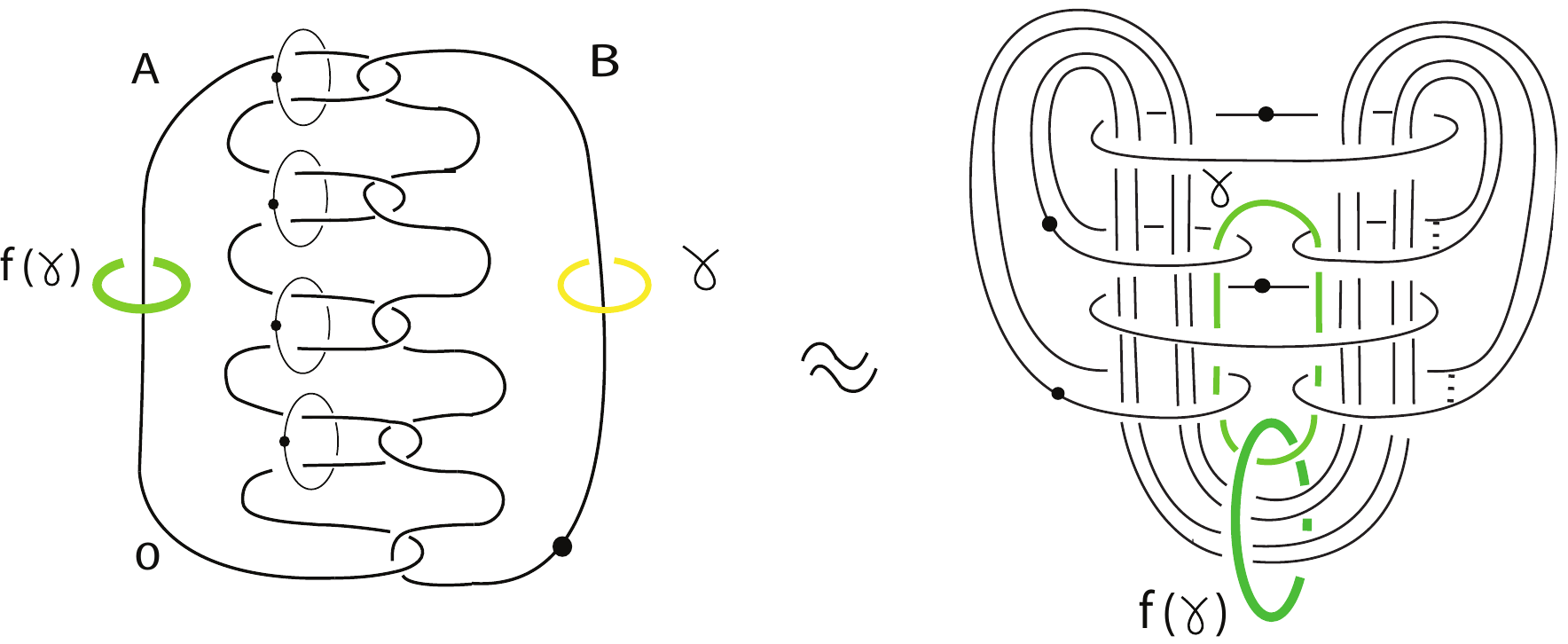}   
\caption{The equivalence $\P_{n} \approx  L_{n}^{0}$, when $n=2$}   \label{a1}
\end{center}
\end{figure}

\begin{figure}[ht]  \begin{center}
\includegraphics[width=.56\textwidth]{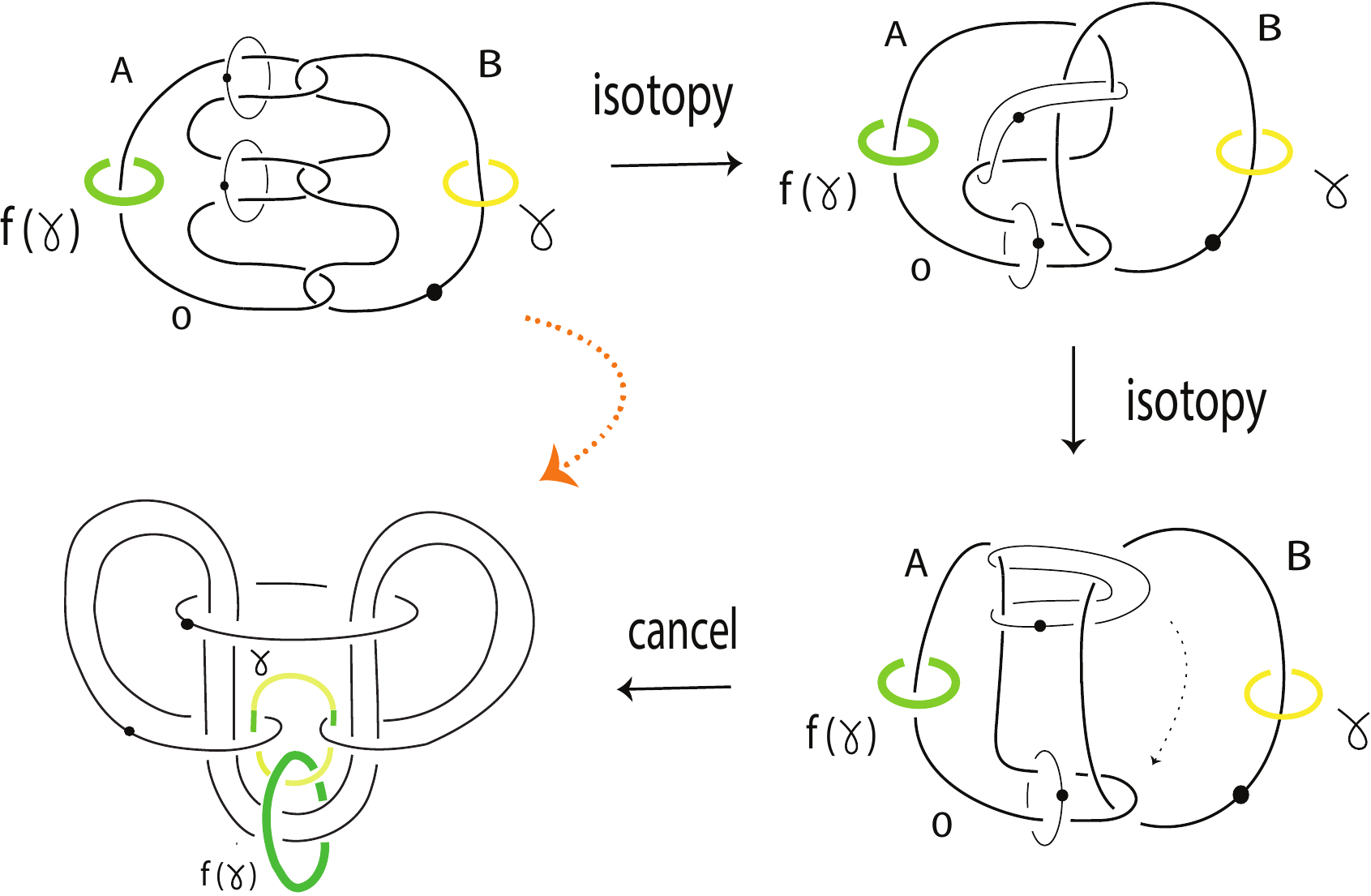}   
\caption{Describing $\P_{1}\approx L^{0}_{1}$}   \label{a6}
\end{center}
\end{figure}

\newpage

In general a protocork  $\V(\CA)$ 
is a union of $\P_{n}$'s where you are allowed to link the dotted circles with the $0$-framed circles algebraically zero times. Protocork twisting map $f$ is  just the ``zero-dot exchanging map'' between $\CA$ and $\CB$ spheres (swapping  $1$- and $2$-handles of $\CA$ and $\CB$). 

\vspace{.1in}

With this, we can identify protocork twistings of $S^{4}$ along the protocorks $\P_{n} \subset S^{4}$, and along $L_{n}^{0}\subset S^{4}$, which by \cite{ay} is diffeomorphic to $S^{4}$. This gives another proof for Theorem~\ref{main} in the case of $\V(\CA)=\P_{n}$. So to prove 4-dim smooth Poincar\'e conjecture we need to show that protocork twisting of $S^{4}$ along any protocork in $S^{4}$ gives back $S^{4}$.

\begin{figure}[ht]  \begin{center}
\includegraphics[width=.66\textwidth]{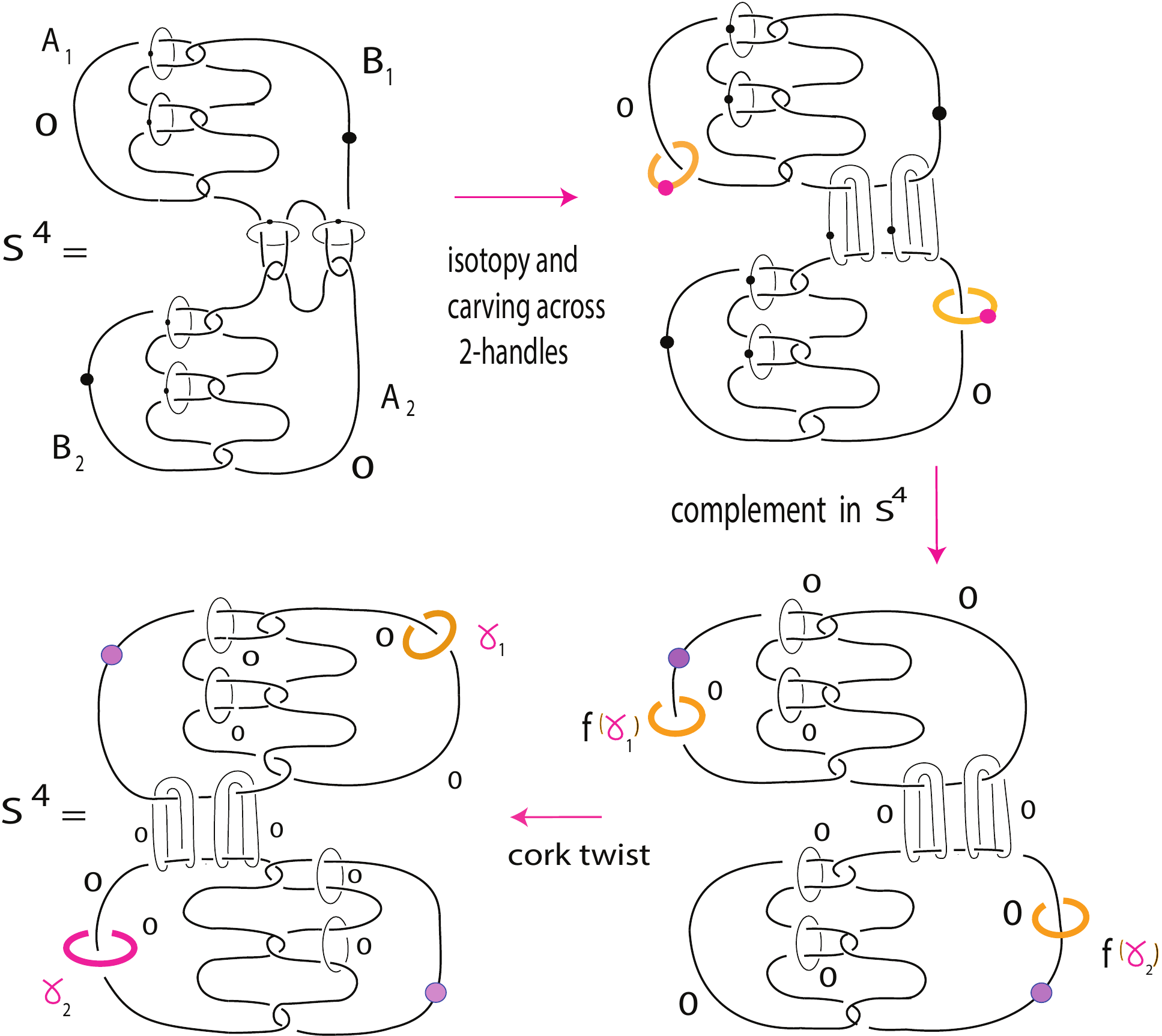}
\caption{General protocork twisting $S^{4}$}   \label{t4}
\end{center}
\end{figure}

 For this we consider ``excess  intersections'' between different $\P_{n}$'s of a protocork $\V(\CA)$, i.e. the case of $\CA$ spheres of  $\V$ intersecting its $\CB$ spheres (algebraically zero times).  The protocol described in the right picture of  Figure~\ref{y1} is a good example to understand the general case; it corresponds to the handlebody in the first picture of Figure~\ref{t4}. Now we will proceed to show how to modify the proof in case $\V(\CA)=\P_{n}$ to this general case.
 
 \vspace{.1in}


 Now denote the protocork complement in $S^{4}$ by $C=S^{4} -\P$, and let $f(\gamma_1),...,f(\gamma_n)$ be the linking circles of the $2$-handles $A_1,...,A_n$ of the protocork $\P$.  Clearly, carving $\P$ along its $2$-handles  (that is carving it along the meridianal discs which $f(\gamma_{i})$'s bound) will turn $\P$ into disjoint union of thickened circles  $\#S^{1}\times B^{3}\hookrightarrow S^{4}$, which is the second picture of Figure~\ref{t4}; hence its complement in $S^{4}$ is $C^{*}:=C\smile$\{2-handles along the curves $f(\gamma_{i})$\}  is a disjoint union of thickened $2$-spheres $\#B^{2}\times S^{2}$ in $S^{4}$.  Hence $S^4 = C^{*} \smile$ $3$-handles attached at top (in Figure~\ref{p1}).

\vspace{.1in}

 Now assume the first handlebody of  Figure~\ref{t4} lies in $S^{4}$, then by an isotopy and carving across its $2$-handles we obtain the second picture. Then by taking its complement in $S^{4}$ we obtain the third picture (the thick dots indicate carving along ``some'' disk bounding the corresponding circles). Then by applying protocork twist to the third picture and proceeding as in  Figure~\ref{y5b} we obtain the last picture which is diffeomorphic to $S^{4}$ by the same argument (after cancelling its $2$-handles with $3$-handles ). \end{proof}

\end{document}